\documentclass[12pt,a4paper]{amsart}
\usepackage[T1]{fontenc}
\usepackage[english]{babel}
\usepackage{amsfonts}
\usepackage{ifthen}
\usepackage{amsmath}
\usepackage{verbatim}
\usepackage{graphicx}
\usepackage{amscd,amssymb,amsthm}

\newcounter{minutes}
\divide\time by 60
\newcounter{hours}
\multiply\time by 60 \addtocounter{minutes}{-\time}

\newcommand{\R}{{\mathbb{R}}}
\newcommand{\B}{{\mathbb{B}}}

\newcommand{\Bn}{{\mathbb{B}^n}}
\newcommand{\Rn}{{\mathbb{R}^n}}
\newcommand{\Hn}{{\mathbb{H}^n}}

\date{}

\title[Local convexity properties of balls]{\bf  Local convexity properties of balls in Apollonian and Seittenranta's metrics}

\author{Riku Kl\'en}

\address{Department of Mathematics and Statistics, University of Turku, 20014 Turku, Finland}
\email{ripekl@utu.fi}

\newtheorem{theorem}[equation]{Theorem}

\newtheorem{lemma}[equation]{Lemma}
\newtheorem{proposition}[equation]{Proposition}
\newtheorem{example}[equation]{Example}
\newtheorem{corollary}[equation]{Corollary}
\newtheorem{remark}[equation]{Remark}
\newtheorem{openprob}[equation]{Open problem}

\numberwithin{equation}{section}

\begin{document}

\begin{abstract}
We consider local convexity properties of balls in the Apollonian and Seittenranta's metrics. Balls in the Apollonian metric are considered in the twice punctured space and starlike domains. Balls in Seittenranta's metric are considered in the twice punctured space and in the punctured ball.
\end{abstract}

\def\thefootnote{}
\footnotetext{
\texttt{\tiny File:~\jobname .tex,
          printed: \number\year-\number\month-\number\day,
          \thehours.\ifnum\theminutes<10{0}\fi\theminutes}
}
\makeatletter\def\thefootnote{\@arabic\c@footnote}\makeatother

\maketitle



\section{Introduction}

During the past few decades the hyperbolic and, more generally, the hyperbolic type distances have been studied by many authors in the context of metric spaces such as the Euclidean and Banach spaces \cite{Klén08a,Klén08b,Klén09,Klén10,KlénRasilaTalponen10,MartioVaisala11,RasilaTalponen12,Vaisala07,Vaisala09}. The purpose of this paper is to study the geometry of balls defined by two M\"obius invariant distances in the Euclidean space.

The first distance, the Apollonian distance was first introduced in \cite{Barbilian34} and later reintroduced in the context of the hyperbolic distance by A.F. Beardon \cite{Beardon98}. The Apollonian distance has recently been studied  as a metric \cite{Hasto03,Hasto04,Hasto07,Ibragimov03}, in connection with quasiconformal mappings \cite{GehringHag00}  and John domains \cite{WangHuangPonnusamyChu07}.

The second distance, Seittenranta's distance was introduced in 1999 by P. Seittenranta \cite{Seittenranta99} and it was based on the observations in \cite{Vuorinen88}. It has also been studied recently in \cite{Hasto04,Hasto07,HastoIbragimovLinden06}.

The cross-ratio $|a,b,c,d|$ for $a,b,c,d \in \overline{\Rn}$ is defined by
\[
  |a,b,c,d| = \frac{|a-c||b-d|}{|a-b||c-d|}.
\]
If $a = \infty$, $c = \infty$ or $d = \infty$ then we define $|\infty,b,c,d| = |b-d|/|c-d|$, $|a,b,\infty,d| = |b-d|/|a-b|$ and $|a,b,c,\infty| = |a-y|/|a-x|$.

Let $G$ be a proper subdomain of $\overline{\Rn}$. The Apollonian distance is defined for $x,y \in G$ by
\[
  \alpha_G(x,y) = \sup_{a,b \in \partial G} \log |a,x,y,b| = \sup_{a,b \in \partial G} \log \frac{|a-y||x-b|}{|a-x||y-b|}.
\]
Note that $\alpha_G$ is a metric if and only if $\partial G$ is not contained in a sphere in $\overline{\Rn}$, \cite[Theorem 1.1]{Beardon98}.

Seittenranta's distance is defined for $x,y \in G \subset \overline{\Rn}$ with $\textrm{card}\, \partial G \ge 2$ by
\[
  \delta_G(x,y) = \sup_{a,b \in \partial G} \log (1+|a,x,b,y|) = \sup_{a,b \in \partial G} \log \left( 1+\frac{|a-b||x-y|}{|a-x||y-b|} \right)
\]
and it is always a metric \cite[Theorem 3.3]{Seittenranta99}.

We shall study here local convexity properties, such as convexity and starlikeness, of the balls defined by the two distances. The question about convexity of hyperbolic type metric balls was posed by M. Vuorinen in 2007 \cite[8.1]{Vuorinen07}.

Our main results are the following theorems:
\begin{theorem}\label{mainthm}
Let $G \subsetneq \Rn$ be a domain such that the complement of $G$ is not contained in any $(n-1)$-dimensional sphere, $x \in G$ and $r > 0$.

(1) Let $x,y \in \Rn$, $x \neq y$, and $G = \Rn \setminus \{ x,y \}$. Then $B_\alpha (z,r)$ is not convex for any $z \in G$ and $r > 0$.

(2) If $G$ starlike with respect to $x$, then $B_\alpha(x,r)$ is strictly starlike with respect to $x$.
\end{theorem}

\begin{theorem}\label{mainthm2}

(1) Let $G = \Bn \setminus \{ 0 \}$, $x \in G$ and $r_0 = \log (1+1/(1-|x|))$. Then $B_\delta (x,r)$ is convex for all $r \in (r,r_0]$ and is not convex for $r > r_0$.

(2) Let $G = \Rn \setminus \{ x_1, \dots ,x_m \}$, $m \ge 2$ and $x_1 \neq x _2$, $x \in G$ and $r > 0$. Then $B_\alpha(x,r)$ is convex, if
  \[
    r \le \log \left( 1+\frac{\displaystyle \min_{i \neq j} \left\{ |x_i-x_j| \right\}}{\displaystyle \max_{i} \{ |x-x_i| \} } \right).
  \]
\end{theorem}

In this paper we will shortly introduce known results and some preliminaries in Section \ref{preliminary}. In Section \ref{apo} we concentrate on the Apollonian metric balls. We consider $B_\alpha(x,r)$ in the twice punctured space $\Rn \setminus \{ a,b \}$ and domains, which are starlike with respect to $x$. In Section \ref{seit} we study Seittenranta's metric balls in twice punctured space and in punctured unit ball $\B^n \setminus \{ 0 \}$.

\section{Preliminary results}\label{preliminary}

A domain $G \subsetneq \Rn$ is starlike with respect to $x \in G$ if for all $y \in G$ the line segment $[x,y]$ is contained in $G$ and $G$ is strictly starlike with respect to $x$ if each half-line from the point $x$ meets $\partial G$ at exactly one point. Clearly (strictly) convex domains are (strictly) starlike with respect to any point.

The cross-ratio is M\"obius invariant, which means that for each M\"obius transformation $f$ we have $|a,b,c,d| = |f(a),f(b),f(c),f(d)|$. Therefore $\alpha_G$ and $\delta_G$ are M\"obius invariant.

For a distance $d$ in $G$ we define the metric ball for $x \in G$ and $r>0$ by $B_d(x,r) = \{ y \in G \colon d(x,y) < r \}$. The Euclidean balls and spheres we denote $B^n(x,r)$ and $S^{n-1}(x,r)$, respectively. We denote the unit ball $B(0,1)$ by $\Bn$ and the upper half-space by $\Hn = \{ z \in \Rn  \colon z_n > 0 \}$. The hyperbolic distance in the unit ball $\Bn$ and in the upper half-space $\Hn$ are denoted by $\rho_\Bn$ and $\rho_\Hn$, respectively.

For $x,y \in \Rn$ and $r >0$ we define the Apollonian ball and sphere, respectively, to be
\[
  B_{x,y}^r = \{ z \in \Rn \colon r |x-z| < |y-z| \}, \quad S_{x,y}^r = \{ z \in \Rn \colon r |x-z| = |y-z| \}.
\]
For $x,y \in \Rn$ and $c > 0$, $c \neq 1$, we have \cite[Lemma 2.2.3]{Ibragimov02}
\begin{equation}\label{Apo-formula}
  S_{x,y}^c =S^{n-1} \left( \frac{y-c^2x}{1-c^2},\frac{c|x-y|}{|1-c^2|} \right).
\end{equation}
In the case $c = 1$ the Apollonian ball is a half-space.

Note that in the definition of $\alpha_G$ and $\delta_G$ we can replace the supremum by maximum, if we additionally allow that $a$ or $b$ may be infinity in the case of unbounded $G$.

\begin{proposition}\cite[Theorem 3.11]{Seittenranta99}
  Let $G \subsetneq \Rn$ be an open set. The for all $x,y \in G$ we have
  \[
    \alpha_G(x,y) \le \delta_G(x,y) \le \log \left( e^{\alpha_G(x,y)} +2 \right) \le \alpha_G(x,y) +\log 3.
  \]
\end{proposition}

From the definition it is also easy to verify that the Apollonian distance is monotone with respect to the domain, i.e. for all $x,y \in G' \subset G$ we have
\begin{equation}\label{Apo:inclusion}
  \alpha_{G}(x,y) \le \alpha_{G'}(x,y).
\end{equation}

The following proposition shows that Seittenranta's distance is also monotone with respect to the domain.

\begin{proposition}\cite[Remark 3.2 (2)]{Seittenranta99}
  Let $G \subsetneq \Rn$ and $G' \subset G$ be domains. Seittenranta's distance is monotone with respect to the domain, \emph{i.e.} for all $x,y \in G'$ we have
  \begin{equation}\label{Seit:inclusion}
    \delta_{G}(x,y) \le \delta_{G'}(x,y).
  \end{equation}
\end{proposition}

We introduce next a result that can be used to estimate metric balls $B_\alpha$ and $B_\delta$.

\begin{theorem}\label{ApoSeit:intersection}
  Let $G \subsetneq \Rn$, $\partial G \subset \overline{\Rn}$, $\textnormal{card} \, \partial G \ge 2$, $x \in G$ and $r>0$. Then for $m \in \{ \alpha , \delta \}$
  \[
    B_{m_G} (x,r) = \bigcap_{a,b \in \partial G \atop a \neq b} B_{m_{\Rn \setminus \{ a,b \}}} (x,r).
  \]
\end{theorem}
\begin{proof}
%
%

  We show that for all $x,y \in G$
  \begin{equation}\label{newformula}
    m_G(x,y) = \sup_{a,b \in \partial G, \, a \neq b} m_{\Rn \setminus \{ a,b \}}(x,y).
  \end{equation}
  Because $G \subset \Rn \setminus \{ a,b \}$ for all $a,b \in \partial G$ by \eqref{Apo:inclusion} and \eqref{Seit:inclusion} we have
  \[
    m_G(x,y) \ge m_{\Rn \setminus \{ a,b \}}(x,y)
  \]
  and thus
  \[
    m_G(x,y) \ge \sup_{a,b \in \partial G, \, a \neq b} m_{\Rn \setminus \{ a,b \}}(x,y).
  \]

  On the other hand, for some $a,b \in \partial G$ with $a \neq b$ we have
  \[
    m_G(x,y) = m_{\Rn \setminus \{ a,b \}}(x,y)
  \]
  and \eqref{newformula} holds.
\end{proof}

Let us fix two distinct points $x,y \in \Rn$ and a radius $r > 1$. Then the union of the Apollonian balls $B_{x,z}^r$ for $z \in [x,y]$ form an ''ice cream cone''. This observation is stated formally in the following lemma.

\begin{lemma}\label{Apo:icecreamcone}
  Let $x,y \in \Rn$ with $x \neq y$ and $r \in (0,1)$. Then
  \[
    \bigcup_{t \in (0,1]} B_{x,z}^r = A \cup B_{x,y}^r,
  \]
  where $z = x+t(y-x)$ and
  \[
    A = \left\{ a \in \Rn \colon \measuredangle (a,x,y) < \arcsin r, |a-x| < \frac{|x-y|}{\sqrt{1-r^2}} \right\}.
  \]
\end{lemma}
\begin{proof}
  We show that
  \[
    \bigcup_{s > 0} B_{x,x+s(y-x)}^r = \left\{ a \in \Rn \colon \measuredangle (a,x,y) < \arcsin r \right\}.
  \]
  For $b \in \Rn$ and $c > 0$ such that $B^n(b,c) = B_{x,x+s(y-x)}^r$ we show that the ratio of $c$ and $|x-b|$ is a constant. By  \eqref{Apo-formula}
  \[
    \frac{c}{|x-b|} = \frac{r|x-(x+s(y-x))|}{|1-r^2||x-b|} = \frac{sr|x-y|}{|1-r^2||x-b|} = \frac{r s |x-y|}{|x(1-r^2)-x-s(y-x)+r^2x|} = r.
  \]
  Thus the union of the Apollonian balls $B_{x,x+s(y-x)}^r$ is an angular domain with $\measuredangle (a,x,y) < \arcsin r$.

  By the Pythagorean theorem and \eqref{Apo-formula}
  \[
    |x-b| \le \sqrt{ \left( x-\frac{y-r^2 x}{1-r^2} \right)^2 + \left( \frac{r|x-y|}{|1-r^2|} \right)^2 } = \frac{|x-y|\sqrt{1-r^2}}{|1-r^2|} = \frac{|x-y|}{\sqrt{1-r^2}}
  \]
  and the assertion follows.
\end{proof}

\begin{example}\label{knownresults}
  (1) By \cite[Lemma 3.1]{Beardon98} we have $\alpha_\Hn = \rho_\Hn$ and since both distances are M\"obius invariant the distances agree in all domains which can be obtained from $\Hn$ by a M\"obius transformation. Especially, we have $\alpha_\Bn = \rho_\Bn$. By \cite[Lemma 8.39]{Vuorinen88} we have $\delta_\Bn = \rho_\Bn$ and by M\"obius invariance it is clear that $\alpha_G = \delta_G$ for $G=f(\Bn)$, where $f$ is a M\"obius transformation. Especially $\alpha_G = \delta_G = \rho_G$ for $G \in \{ \Bn,\Hn \}$. By \cite[(2.11)]{Vuorinen88}
  \[
    B_\alpha(x,r) = B_\delta(x,r) = B^n \left( x+x_n e_n (\cosh r -1),x_n \sinh r \right)
  \]
  for all $x \in \Hn$, $r > 0$ and by \cite[(2.22)]{Vuorinen88}
  \[
    B_\alpha(x,r) = B_\delta(x,r) = B^n \left( \frac{x(1-\tanh^2 \frac{r}{2})}{1-|x|^2 \tanh^2\frac{r}{2}},\frac{(1-|x|^2)\tanh \frac{r}{2}}{1-|x|^2 \tanh^2\frac{r}{2}} \right)
  \]
  for all $x \in \Bn$, $r > 0$.

  (2) Let $G = \Rn \setminus \{ a \}$, where $a \in \Rn$. Then by \cite[Remark 3.2 (3)]{Seittenranta99} $\delta_G = j_G$, where $j_G$ is the distance-ratio metric defined by
    \[
      j_G(x,y) = \log \left( 1+\frac{|x-y|}{\min \{ d(x,\partial G),d(y,\partial G) \} } \right) = \log \left( 1+\frac{|x-y|}{\min \{ |x-a|,|y-a| \} } \right)
    \]
    for all $x,y \in G$. By \cite[proof of Theorem 3.1]{Klén08a} and \eqref{Apo-formula} we have for all $x \in \Rn \setminus \{a\}$
  \[
    B_\delta(x,r) = \left\{ \begin{array}{ll} B^n(x,(e^r-1)|x-a|) \cap B_{x,a}^{1/(e^r-1)}, & \textrm{for } 0 < r \le \log 2, \\ B^n(x,(e^r-1)|x-a|) \setminus \overline{B}_{x,a}^{1/(e^r-1)}, & \textrm{for }r > \log 2. \end{array} \right.
  \]
\end{example}

By Example \ref{knownresults}, \cite[Theorem 3.1]{Klén08a} and \cite[Theorem 3.4]{Klén08a} we collect the following result.
\begin{proposition}
(1) Let $G \in \{ \Bn,\Hn \}$ and $x \in G$. Then $B_\alpha(x,r)$ and $B_\delta(x,r)$ are strictly convex for all $r > 0$.

(2) Let $a \in \Rn$,  $G=\Rn \setminus \{ a \}$, $x \in G$, $r_c = \log 2$ and $r_s = \log (1+\sqrt{2})$. Then $B_\delta(x,r)$ is (strictly) convex for $r \in (0,r_c]$ ($r \in (0,r_c)$) and (strictly) starlike with respect to $x$ for $r \in (0,r_s]$ ($r \in (0,r_s)$).
\end{proposition}

\section{Balls in the Apollonian metric}\label{apo}

By the definition we have
\begin{equation}\label{ApollonianBallsAproach}
  \alpha_G(x,y) = \sup_{a \in \partial G} \log \frac{|a-x|}{|a-y|} + \sup_{b \in \partial G} \log \frac{|b-x|}{|b-y|}, 
\end{equation}
which geometrically means that maximal Apollonian balls $B_{x,y}^{r_{xy}} , B_{y,x}^{r_{yx}} \subset G$ determine the Apollonian distance $\alpha_G(x,y) = \log (r_{xy} r_{yx})$.

The supremum in the definition of $\alpha$ is obtained only when $G$ is contained in a half-space $H$ and there exists $G' \subset G$ such that $G' \subset \partial H$ and $\textrm{diam}\, G' = \infty$.

We consider next Apollonian distance in the domain $G = \Rn \setminus \{ -e_1,e_1 \}$. Note that $\alpha_G$ is not a metric in this domain. Especially, for $x \in G$ and $a = |x+e_1|/|x-e_1|$ we have $\alpha_G(x,y) = 0$ for
\begin{equation}\label{Apo=0}
  y \in S_{e_1,-e_1}^a, \textrm{ if } |x-e_1| \le |x+e_1|, \textrm{ and } y \in S_{-e_1,e_1}^{1/a}, \textrm{ if } |x-e_1| \ge |x+e_1|.
\end{equation}

\begin{theorem}\label{Apo:twopoints}
  Let $G = \Rn \setminus \{ -e_1,e_1 \}$, $x \in G$ and $r > 0$. We denote
  \[
    B_c = B_{e_1,-e_1}^c, \quad B_d =B_{-e_1,e_1}^d
  \]
  for $c = e^r |x+e_1|/|x-e_1|$ and $d = e^r |x-e_1|/|x+e_1|$. Then
  \[
    B_\alpha (x,r) = \left\{ \begin{array}{ll}B_c \setminus \overline{B_d}, & \textrm{if } c<1 \textrm{ and } d \ge 1,\\ \Rn \setminus ( \overline{B_c} \cup \overline{B_d} ), & \textrm{if } c>1 \textrm{ and } d > 1,\\ B_d \setminus \overline{B_c}, & \textrm{if } c \ge 1 \textrm{ and } d < 1. \end{array} \right.
  \]
  Moreover, the complement of $B_\alpha(x,r)$ is always disconnected.
\end{theorem}
\begin{proof}
   By definition $\alpha_G(x,y) =r$ is equivalent to
  \begin{equation}\label{apocircles}
    c |y-e_1| = |y+e_1| \quad \textrm{or} \quad d |y+e_1| = |y-e_1|.
  \end{equation}
  Equalities \eqref{apocircles} determine Apollonian spheres with respect to points $e_1$ and $-e_1$ and by \eqref{Apo-formula} the Apollonian spheres are $S_{e_1,-e_1}^c$ and $S_{-e_1,e_1}^d$.
  We denote $S_c = S_{e_1,-e_1}^c$ and $S_d = S_{-e_1,e_1}^d$. By \eqref{Apo=0} we see that $S_a = S_{e_1,-e_1}^a$
, where $a = |x+e_1|/|x-e_1|$, is contained in $B_\alpha(x,r)$. Note that all the spheres $S_a$, $S_c$ and $S_d = S_{-e_1,e_1}^d =S_{e_1,-e_1}^{1/d}$ are Apollonian spheres and $1/d < a < c$. Since $S_a \subset B_\alpha(x,r)$ and $\partial B_\alpha (x,r) = S_c \cap S_d$, the complement of $B_\alpha(x,r)$ is disconnected. We denote the convex hull of $S_a$ by $B_a$. 

  Let us assume that $c<1$ and $d \ge 1$. Now also $a > 1$. Because
  \[
    \frac{1+c^2}{c^2-1}+\frac{2c}{c^2-1} < \frac{1+a^2}{a^2-1} + \frac{2a}{a^2-1}
  \]
  is equivalent to $(a e^r+1)/(a e^r-1) < (a+1)/(a-1)$ and
  \[
    \frac{1+c^2}{c^2-1}-\frac{2c}{c^2-1} > \frac{1+a^2}{a^2-1} - \frac{2a}{a^2-1}
  \]
  is equivalent to $(a e^r-1)/(a e^r+1) < (a-1)/(a+1)$, it is clear that $S_c \subset S_a$. A similar argument shows that $B_a \subset B_d$.

  Let us then assume that $c>1$ and $d>1$. It is easy to verify that $B_d \subset (-\infty,0) \times \R^{n-1}$ and $B_c \subset (0,\infty) \times \R^{n-1}$. Thus we have $B_c \cap B_d = \emptyset$. Since $S_a$, $S_c$ and $S_d$ are Apollonian spheres with $1/d < a < c$ it is clear that $S_a \cap (B_c \cup B_d) = \emptyset$.

  The case  $c \ge 1$ and $d < 1$ is proved similarly to the case  $c<1$ and $d \ge 1$.
\end{proof}

Examples of Theroem \ref{Apo:twopoints} in twice punctured plane are represented in Figure \ref{Apo:kuva1}.

\begin{figure}[ht]
\begin{center}
      \includegraphics[height=.3\textwidth]{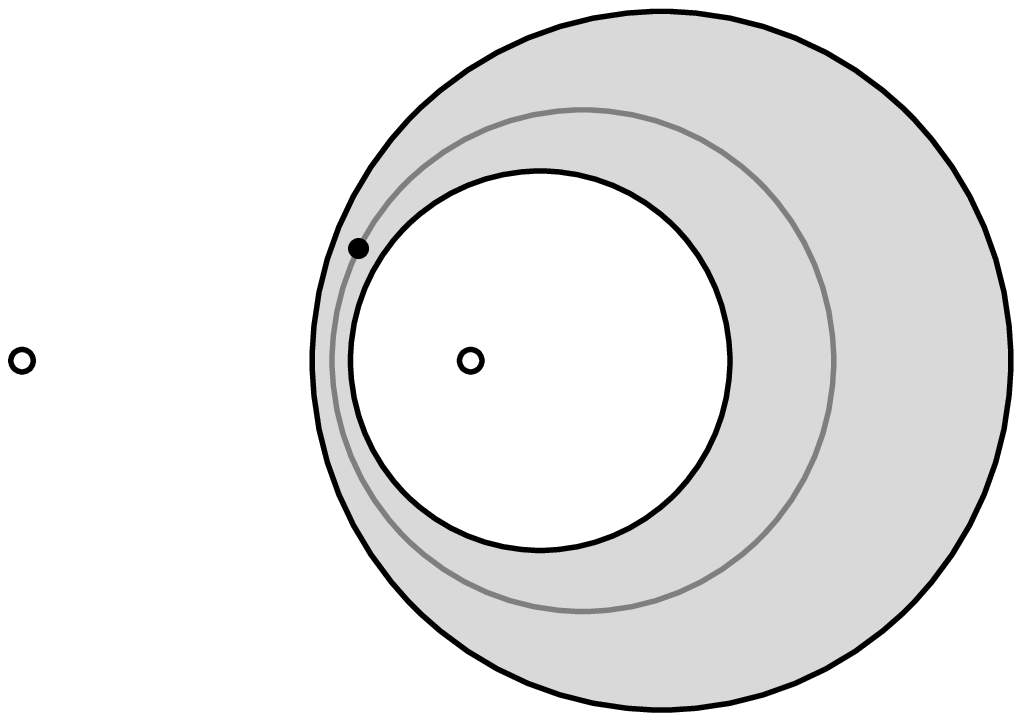}\hspace{5mm}
      \includegraphics[height=.3\textwidth]{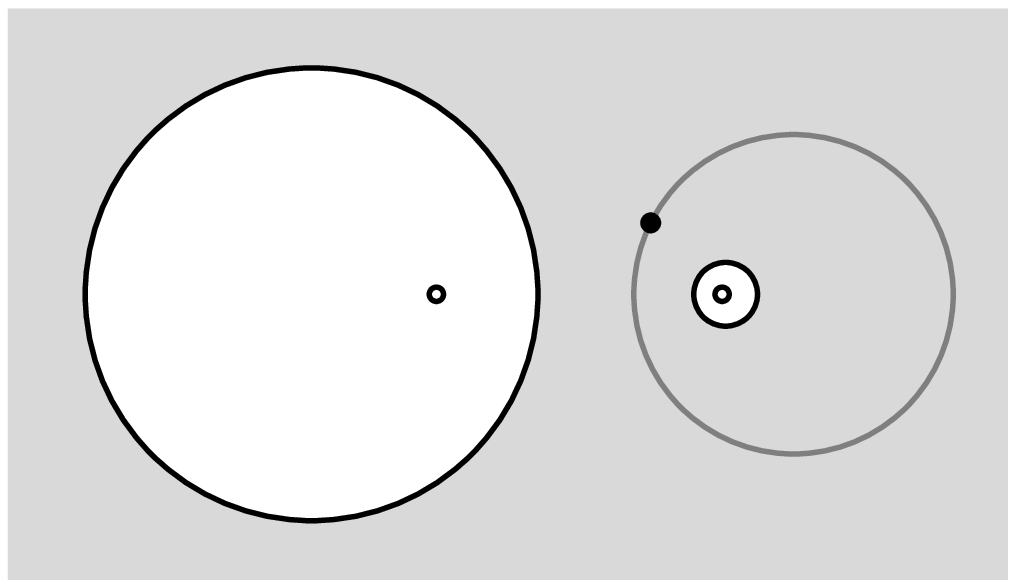}
    \end{center}
    \caption{Disks $B_\alpha(x,r)$ in the domain $\R^2 \setminus \{ 1,-1 \}$ with $x=(1+i)/2$, $r = 1/5$ (on the left) and $r=7/5$ (on the right). The black dot is $x$, the gray circle is the set defined by \eqref{Apo=0} and the small black circles are $1$ and $-1$.\label{Apo:kuva1}}
\end{figure}

\begin{remark}\label{Apo:remark} (1) Theorem \ref{Apo:twopoints} can be generalized for any twice punctured space:

  Let $y,z \in \Rn$ with $y \neq z$, $G = \Rn \setminus \{ y,z \}$, $x \in G$ and $r > 0$. We denote
  \[
    B_c = B_{y,z}^c, \quad B_d = B_{z,y}^d
  \]
  for $c = e^r |x-z|/|x-y|$ and $d = e^r |x-y|/|x-z|$. Then
  \[
    B_\alpha (x,r) = \left\{ \begin{array}{ll}B_c \setminus \overline{B_d}, & \textrm{if } c<1 \textrm{ and } d \ge 1,\\ \Rn \setminus ( \overline{B_c} \cup \overline{B_d} ), & \textrm{if } c>1 \textrm{ and } d > 1,\\ B_d \setminus \overline{B_c}, & \textrm{if } c \ge 1 \textrm{ and } d < 1. \end{array} \right.
  \]
  Moreover, the complement of $B_\alpha(x,r)$ is always disconnected.

  (2) By Theorem \ref{ApoSeit:intersection} and (1) we can find a formula for the Apollonian metric balls in the domain $\Rn \setminus G$, where $G = \{ x_1, \dots , x_m \}$ with $m \ge 2$ and $x_1 \neq x_2$.
\end{remark}

We consider next Apollonian metric balls in starlike domains $G \subsetneq \Rn$. In convex domains the Apollonian distance is always a metric.

\begin{theorem}\label{Apo:starlike}
  Let $G \subsetneq \Rn$ be a starlike domain with respect to $x \in G$ such that the complement of $G$ is not contained in any $(n-1)$-dimensional sphere and $r > 0$. Then $B_\alpha(x,r)$ is strictly starlike with respect to $x$.
\end{theorem}
\begin{proof}
  Let us assume that $B_\alpha(x,r)$ is not starlike with respect to $x$. Then there exists $y,z \in G$ such that $y$ is contained in the line segment $(x,z)$, $\alpha_G(x,z) < r$ and $\alpha_G(x,y) =r' \ge r$. Now $B_{x,y}^{r'} \subset G$ and $S_{x,z}^r$ contains a point on $\partial G$. By Lemma \ref{Apo:icecreamcone} this is a contradiction.
\end{proof}

\begin{proof}[Proof of Theorem \ref{mainthm}]
  The assertion follows from Theorem \ref{Apo:twopoints}, Remark \ref{Apo:remark} (1) and Theorem \ref{Apo:starlike}.
\end{proof}

\begin{openprob}

  (1) If $G \subsetneq \Rn$ is a convex domain and $x \in G$, is $B_\alpha(x,r)$ convex for all $r>0$?

  (2) Let $G=\Bn \setminus \{ 0 \}$ and $x \in G$. Does there exists $r_0 = r_0(|x|) > 0$ such that $B_\alpha(x,r)$ is convex for all $r \in (0,r_0]$?
\end{openprob}

\section{Balls in Seittenranta's metric}\label{seit}

We consider next Seittenranta's distance in the domain $G = \Rn \setminus \{ -e_1,e_1 \}$.

\begin{lemma}\label{Seit:twicePP}
  Let $G = \Rn \setminus \{ -e_1,e_1 \}$, $x \in G$ and $r > 0$. Then for $B_c = B_{-e_1,x}^c$ and $B_d = B_{e_1,x}^d$ we have
  \[
    B_\delta (x,r) = \left\{ \begin{array}{ll} B_c \cap B_d, & \textrm{if } c \le 1 \textrm{ and } d \le 1,\\
    B_c \setminus \overline{B_d}, & \textrm{if } c \le 1 \textrm{ and } d > 1,\\
    B_d \setminus \overline{B_c}, & \textrm{if } c > 1 \textrm{ and } d \le 1,\\
    \Rn \setminus ( \overline{B_c} \cup \overline{B_d}), & \textrm{if } c > 1 \textrm{ and } d > 1, \end{array} \right.
  \]
  where $c = |x-e_1|(e^r-1)/2$ and $d = |x+e_1|(e^r-1)/2$.
\end{lemma}
\begin{proof}
  Let us denote $\beta = |x-e_1|/|x+e_1|$. Since $|x-e_1||y+e_1| < |x+e_1||y-e_1|$ is equivalent to $y \in B_{-e_1,e_1}^\beta$ we have by definition
  \[
    \delta(x,y) = \left\{ \begin{array}{ll} \displaystyle \log \left( 1+\frac{2|x-y|}{|x-e_1||y+e_1|} \right) , & \textrm{if } y \in B_{-e_1,e_1}^\beta,\\
    \displaystyle  \log \left( 1+\frac{2|x-y|}{|x+e_1||y-e_1|} \right), & \textrm{if } y \in B_{e_1,-e_1}^{1/\beta}. \end{array} \right.
  \]
  For $y \in B_{-e_1,e_1}^\beta$ the equality $\delta_G(x,y) = r$ is equivalent to $y \in S_{-e_1,x}^c$. Similarly, for $y \in B_{e_1,-e_1}^{1/\beta}$ the equality $\delta_G(x,y) = r$ is equivalent to $y \in S_{e_1,x}^d$. Therefore it is clear that $\partial B_\delta (x,r) \subset S_{-e_1,x}^c \cup S_{e_1,x}^d$.

  By \eqref{Apo-formula} we can see that $c \le 1$ is equivalent to $x \in B_c$ and $c > 1$ is equivalent to $x \notin \overline{B_c}$. Similarly we observe that $d \le 1$ is equivalent to $x \in B_d$ and $d > 1$ is equivalent to $x \notin \overline{B_d}$. Since always $x \in B_\delta (x,r)$, the above observations imply the assertion.
\end{proof}

Examples of Lemma \ref{Seit:twicePP} in twice punctured plane are represented in Figure \ref{Seit:kuva1}.

\begin{figure}[ht]
\begin{center}
      \includegraphics[height=.3\textwidth]{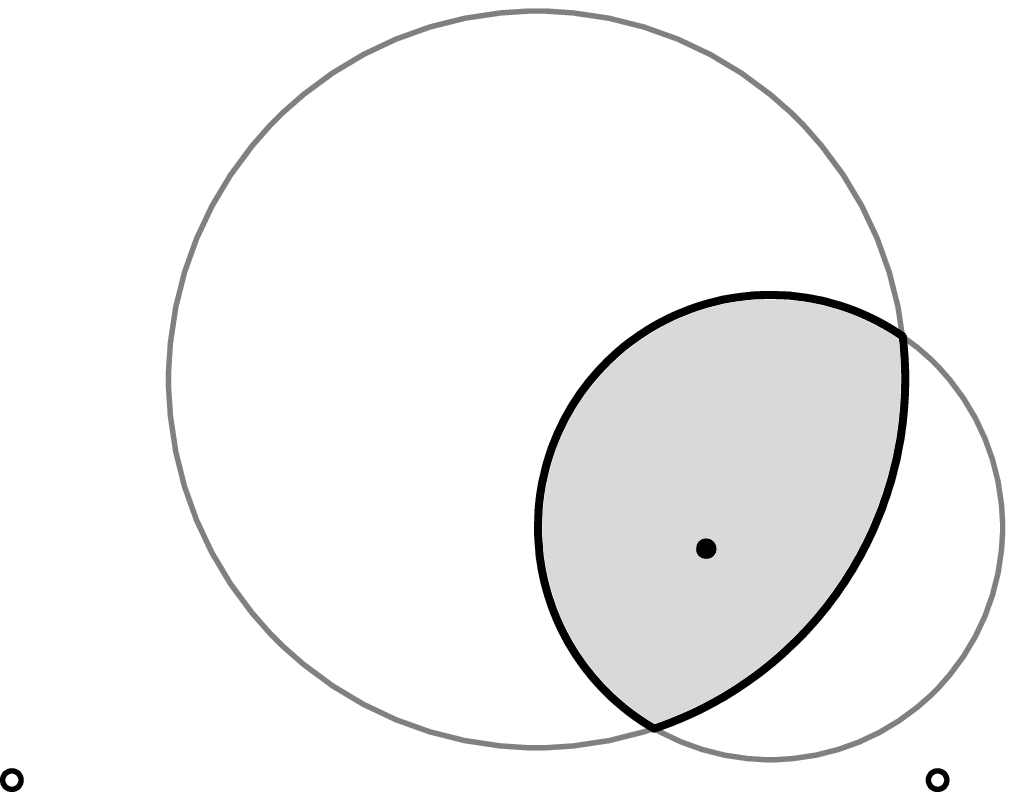}\hspace{5mm}
      \includegraphics[height=.3\textwidth]{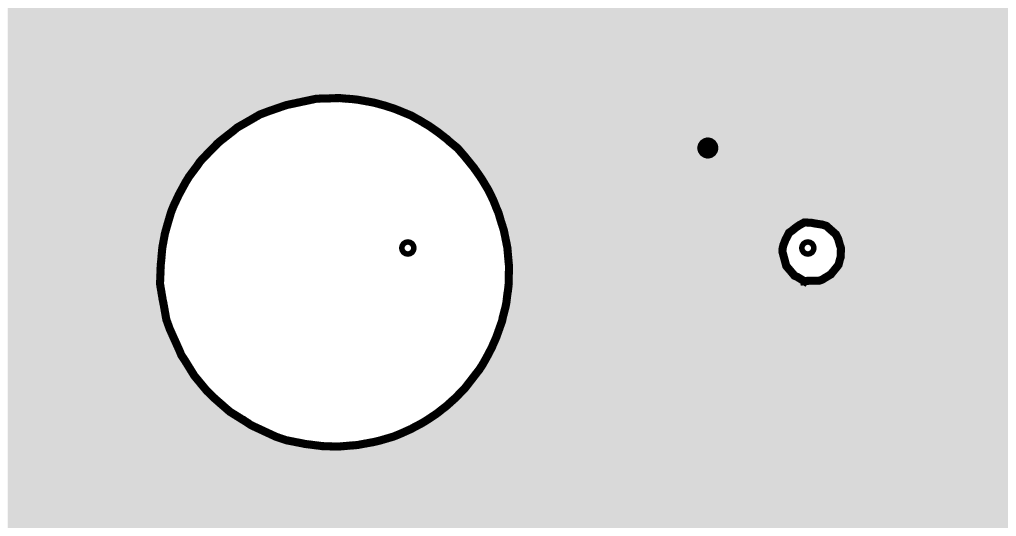}
    \end{center}
    \caption{Disks $B_\delta(x,r)$ in the domain $\R^2 \setminus \{ 1,-1 \}$ with $x=(1+i)/2$, $r = 3/5$ (on the left) and $r=2$ (on the right). The black dot is $x$, the gray circles are $\partial B_c$ and $\partial B_d$ as defined in Lemma \ref{Seit:twicePP} and the small black circles are $1$ and $-1$.\label{Seit:kuva1}}
\end{figure}

\begin{theorem}\label{Seit:convexity}
  Let $G = \Rn \setminus \{ -e_1,e_1 \}$, $x \in G$ and $r_0 = \log(1+2/\max \{ |x-e_1| , |x+e_1| \})$. Then $B_\delta (x,r)$ is convex for all $r \in (0,r_0]$ and is not convex for $r > r_0$.
\end{theorem}

\begin{proof}
  By Lemma \ref{Seit:twicePP} the metric ball $B_\delta (x,r)$ is convex if and only $c \le 1$ and $d \le 1$, which is equivalent to
  \[
    r \le \min \left\{ \log \left( 1+ \frac{2}{|x-e_1|} \right) , \log \left( 1+ \frac{2}{|x+e_1|} \right) \right\} = \log \left( 1+\frac{2}{\max \{ |x-e_1|,|x+e_1| \}} \right)
  \]
  and the assertion follows.
\end{proof}

\begin{remark}\label{Seit:huomautus}
  (1) Theorem \ref{Seit:convexity} is true for any domain $G = \Rn \setminus \{ y,z \}$ with $y,z \in \Rn$ and $a \neq b$, if we replace $r_0$ by
  \[
    r _1 = \log \left( 1+\frac{|y-z|}{\max \{ |x-y|,|x-z| \} } \right).
  \]

  (2) In Theorem \ref{Seit:convexity} (and the above generalization) the radius $r_0$ ($r_1$) is sharp in the sense that for $r \in (0,r_0)$ ($r \in (0,r_1)$) the metric balls $B_\delta (x,r)$ are strictly convex.

  (3) Note that $B_\delta (x,r)$ is not starlike for $r > r_0 (r_1)$ in Theorem \ref{Seit:convexity} (in the above remark (2) ).
\end{remark}

\begin{corollary}\label{Seit:Rn-finiteset}
  Let $G = \Rn \setminus \{ x_1, \dots ,x_m \}$, $m \ge 2$ and $x_1 \neq x _2$, $x \in G$ and $r > 0$. Then $B_\alpha(x,r)$ is convex, if
  \[
    r \le \min_{i \neq j} \left\{ \log \left( 1+\frac{|x_i-x_j|}{\max \{ |x-x_i|,|x-x_j| \} } \right) \right\} =  \log \left( 1+\frac{\displaystyle \min_{i \neq j} \left\{ |x_i-x_j| \right\}}{\displaystyle \max_{i} \{ |x-x_i| \} } \right).
  \]
\end{corollary}
\begin{proof}
  The assertion follows from the fact that intersection of convex domains is convex, Theorem \ref{ApoSeit:intersection} and Remark \ref{Seit:huomautus} (1).
\end{proof}

Note that the radii $r_0$ in Theorem \ref{Seit:convexity} and $r_1$ in Remark \ref{Seit:huomautus} (1) are sharp, but the radius in Corollary \ref{Seit:Rn-finiteset} is not sharp in general. An example of Corollary \ref{Seit:Rn-finiteset} is represented in Figure \ref{Seit:kuva3}.

\begin{lemma}\label{Seit:specialcase}
  Let $x \in \Bn \setminus \{ 0 \}$ and $r >0$. Then the set
  \[
    A = \left\{ y \in \Bn \setminus \{ 0 \} \colon \log \left( 1 +\frac{|x-y|}{|y| (1-|x|)} \right) < r \right\}
  \]
  is convex for $r \in (0,\log(1+1/(1-|x|))]$ and not convex for $r > \log(1+1/(1-|x|))$, and the set
  \[
    B = \left\{ y \in \Bn \setminus \{ 0 \} \colon \log \left( 1 +\frac{|x-y|}{|x| (1-|y|)} \right) < r \right\}
  \]
  is strictly convex.
\end{lemma}
\begin{proof}
  By symmetry it is sufficient to consider only the case $n=2$.

  We prove first the claim for the set $A$. The relation $y \in A$ is equivalent to $\log ( 1 +(|x-y)/(|y| (1-|x|))  < r$, which is equivalent to
  \begin{equation}\label{Seit:Alemma}
    c |x-y| < |y|, \textrm{ where } c = \frac{1}{(e^r-1)(1-|x|)}.
  \end{equation}
  By \eqref{Seit:Alemma} the set $A = B_{x,0}^c$ and by \eqref{Apo-formula} it is convex if and only if $c \ge 1$. Because $c \ge 1$ is equivalent to $r \le \log (1+1/(1-|x|))$ the assertion for the set $A$ follows.

  We prove then the claim for the set $B$. Let $y \in \partial B$. The equality $\log(1+|x-y|/(|x|(1-|y|))) = r$ is equivalent to
  \begin{equation}\label{Siet:mody1}
    |y| = 1- \frac{|x-y|}{c},
  \end{equation}
  where $c = |x|(e^r-1)$. We denote $\beta = \measuredangle (e_1,x,y) \in [0,\pi]$. By the law of cosines we have
  \begin{equation}\label{Siet:mody2}
    |y|^2 = |x-y|^2+|x|^2-2|x-y| |x| \cos (\pi-\beta).
  \end{equation}
  By combining \eqref{Siet:mody1} and \eqref{Siet:mody2} we obtain
  \[
    |x-y| = \left\{ \begin{array}{ll} \displaystyle \frac{1-|x|^2}{2(1-|x|\cos\beta)}, & \textrm{if } c=1,\\ \displaystyle \frac{c \left( 1+c|x|\cos \beta - \sqrt{c^2+|x|^2-c^2 |x|^2 + c|x|(2+c|x| \cos \beta )\cos\beta} \right)}{1-c^2}, & \textrm{if } c \neq 1, \end{array} \right.
  \]
  and we denote $f(\beta) = |x-y|$, if $c=1$, and $g(\beta) = |x-y|$, if $c \neq 1$. We show that $f(\beta)$ and $g(\beta)$ are increasing, which implies that $B$ is strictly convex. 

  We obtain easily that
  \[
    f'(\beta) = \frac{|x|(1-|x|^2) \sin \beta}{2(1+|x| \cos \beta)^2} \ge 0
  \]
  and therefore $f(\beta)$ is increasing.

  By a straightforward computation we get
  \[
    g'(\beta) = c^2|x| \sin \beta \frac{h(\beta)}{c^2-1}, \quad h(\beta) = 1-\frac{1+c |x| \cos \beta}{\sqrt{|x|^2-c^2(|x|^2-1)+c|x| (2+c|x|\cos \beta) \cos \beta}}.
  \]
  Since $h(\beta) > (<) 0$ is equivalent to $c^2-1 > (<) 0$ we conclude that $g'(\beta) \ge 0$ and thus $g(\beta)$ is increasing.
\end{proof}

\begin{theorem}\label{Seit:puncturedball}
  Let $G = \Bn \setminus \{ 0 \}$, $x \in G$ and $r_0 = \log (1+1/(1-|x|))$. Then $B_\delta (x,r)$ is convex for all $r \in (r,r_0]$ and is not convex for $r > r_0$.
\end{theorem}
\begin{proof}
  Let $y,z \in G$, $y \neq z$, and denote by $C$ the circle (or line, if $y$ and $z$ lie on the same diameter) that contains $y$ and $z$ and is perpendicular to $\partial \Bn$. Now $l = C \cap \Bn$ is the hyperbolic line with $y,z \in l$. We denote $\{ y^*,z^* \} = C \cap \partial \Bn$ and assume that $|y-y^*| < |z-y^*|$. Now we have
  \begin{eqnarray}
    \delta_G(y,z) & = & \max \left\{ \max_{a \in \partial \Bn} \log \left( 1+\frac{|a||y-z|}{|y-a||z|} \right) , \max_{b \in \partial \Bn} \log \left( 1+\frac{|b||y-z|}{|y||z-b|} \right) , \rho_\Bn (y,z) \right\} \nonumber\\
    & = & \log \left( 1+|y-z| \max \left\{ \frac{1}{|z|(1-|y|)} , \frac{1}{|y|(1-|z|)} ,\frac{|y^*-z^*|}{|y-y^*||z-z^*|} \right\} \right)\label{Seit:deltaestimate}
  \end{eqnarray}
  and therefore
  \[
    B_{\delta_G} (x,r) = A \cap B \cap C,
  \]
  where $A$ and $B$ as in the Lemma \ref{Seit:specialcase} and $C = B_{\delta_\Bn}(x,r)$. By Lemma \ref{Seit:specialcase} and Example \ref{knownresults} (1) both $B$ and $C$ are always convex. Since $A$ is convex for $r \in (0,r_0]$ by Lemma \ref{Seit:specialcase}, also $B_{\delta_G} (x,r)$ is convex as intersection of three convex domains.

  Finally, we show that the radius $r_0$ is sharp. We denote $y = \partial B_{\delta_G} (x,r) \cap l$, where $l$ is the line segment from $x$ to the origin. We show that for small $\varepsilon$ we have $B^n(y,\varepsilon) \cap B_{\delta_G} (x,r) = B^n(y,\varepsilon) \cap A$, which implies by Lemma \ref{Seit:specialcase} that $B_{\delta_G} (x,r)$ is not convex. We denote $r_A = \log (1+|x-y|/(|y|(1-|x|)))$, $r_B = \log (1+|x-y|/(|x|(1-|y|)))$ and $r_C = \delta_\Bn (x,y)$. We show that $r_A > \max \{ r_B , r_C \}$, which implies the sharpness of $r_0$. Inequality $r_A > r_B$ is equivalent to $|x| > |y|$, which is true by the selection of $y$. Because $r_C = \log(1+2|x-y|/((1+|y|)(1-|x|)))$ it is easy to se that $r_A > r_C$ is equivalent to $|y| < 1$, which is true as $y \in G$.
\end{proof}

An example of Theorem \ref{Seit:puncturedball} is represented in Figure \ref{Seit:kuva3}.

\begin{figure}[ht]
\begin{center}
      \includegraphics[height=.33\textwidth]{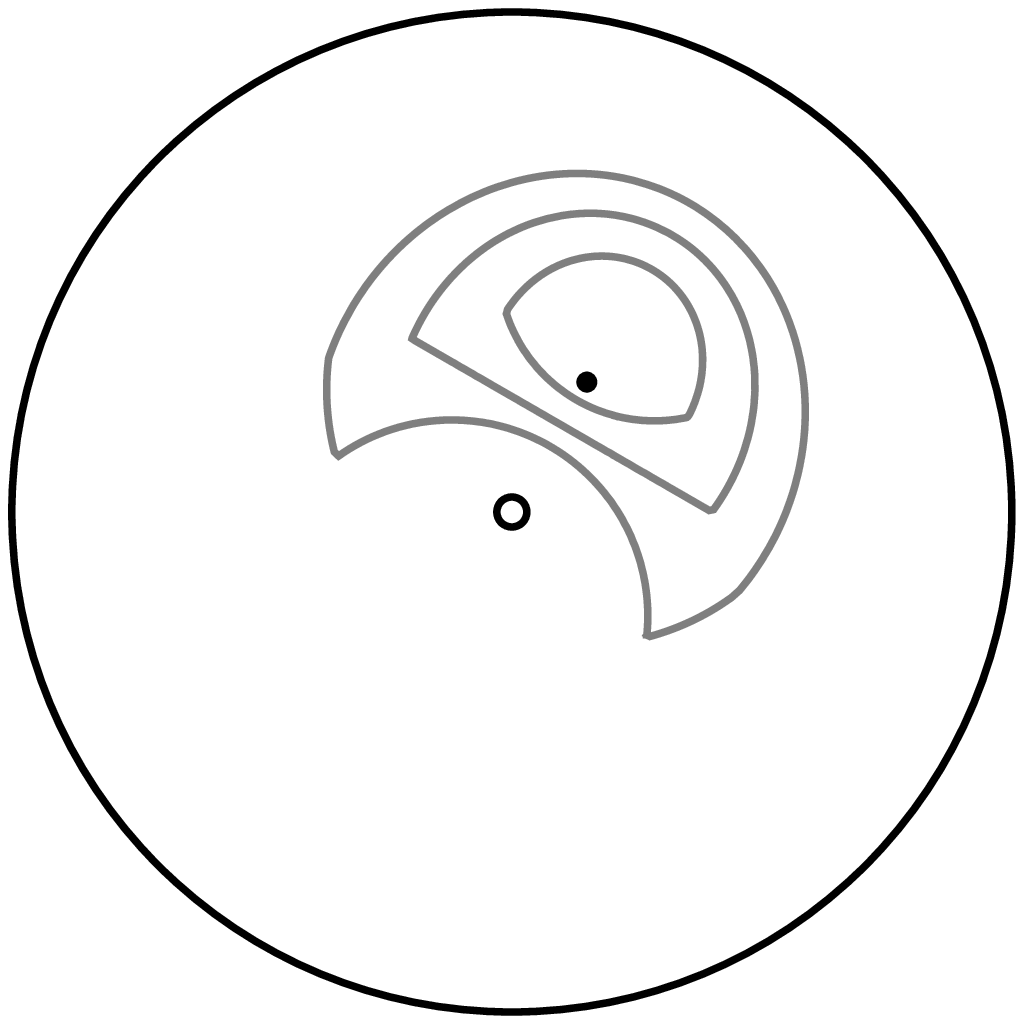}\hspace{2cm}
      \includegraphics[height=.33\textwidth]{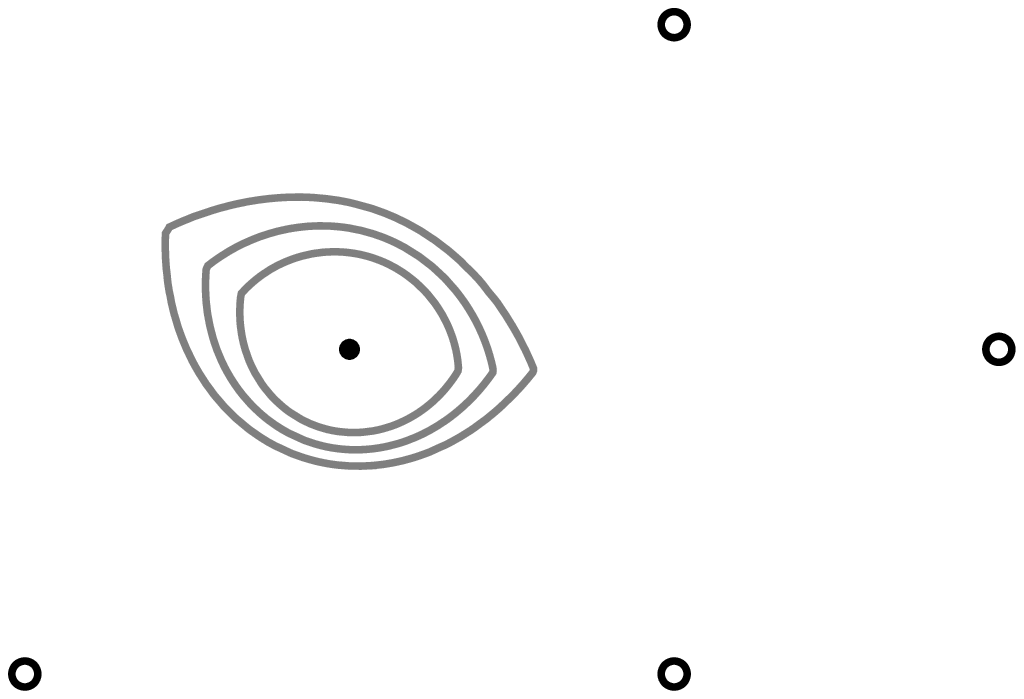}
    \end{center}
    \caption{Disks $B_\delta(x,r)$ of Seittenranta's metric in the domains $\B^2 \setminus \{ 0 \}$ (on the left) and $\R^2 \setminus \{ 1,-1,2+i,1+2i \}$ (on the right) with $r \in \{ r_0-1/3,r_0,r_0+1/3 \}$, where $r_0$ as in Corollary \ref{Seit:Rn-finiteset} on the left and $r_0$ as in Theorem \ref{Seit:puncturedball} on the right. In each figure the black circles form the boundary of the domain and the black dot is the point $x$.\label{Seit:kuva3}}
\end{figure}

\begin{proof}[Proof of Theorem \ref{mainthm2}]
  The assertion follows from Theorem \ref{Seit:puncturedball} and Corollary \ref{Seit:Rn-finiteset}.
\end{proof}

\begin{openprob}

(1) If $G \subsetneq \Rn$ is a convex domain and $x \in G$, is $B_\delta(x,r)$ convex for all $r>0$?

(2) If $G \subsetneq \Rn$ is starlike domain with respect to $x \in G$, is $B_\delta(x,r)$ starlike with respect to $x$ for all $r>0$?
\end{openprob}

{\bf Acknowledgements.} The author would like to thank the referee for useful comments and Matti Vuorinen for posing the open problem and his fruitful ideas.


\end{document}